\documentclass[11pt,reqno]{amsart}

\usepackage{amsmath, amsthm, amssymb}
\usepackage{amsfonts}
\usepackage[ansinew]{inputenc}
\usepackage[dvips]{epsfig}
\usepackage{graphicx}
\usepackage[english]{babel}
\usepackage{comment}
\usepackage{graphicx}
\usepackage{hyperref}
\usepackage{bbm} % for the characteristic function \mathbbm{1}
\usepackage{cite}
\usepackage{graphicx}
\usepackage{amscd}
\usepackage{xcolor}
\usepackage{bm}
\usepackage{enumerate}

\usepackage{verbatim}
\usepackage{hyperref}
\usepackage{amstext}
\usepackage{latexsym}
\let\oldsqrt\sqrt
\def\sqrt{\mathpalette\DHLhksqrt}
\def\DHLhksqrt#1#2{%
\setbox0=\hbox{$#1\oldsqrt{#2\,}$}\dimen0=\ht0
\advance\dimen0-0.2\ht0
\setbox2=\hbox{\vrule height\ht0 depth -\dimen0}%
{\box0\lower0.4pt\box2}}

\allowdisplaybreaks

\newcommand{\R}{\mathbb{R}} % reelle Zahlen
\newcommand{\ov}{\overline}

\renewcommand{\phi}{\varphi}
\renewcommand{\div}{\textnormal{div}}

\newcommand{\cL}{{\mathcal L}}

\newcommand{\cN}{{\mathcal N}}

\newcommand{\cQ}{{\mathcal Q}}

\newcommand{\cX}{{\mathcal X}}

\theoremstyle{definition}
\newtheorem{defi}{Definition}[section]
\newtheorem{remark}[defi]{Remark}

\theoremstyle{plain} %default%plain
\newtheorem{thm}[defi]{Theorem}

\newtheorem{lemma}[defi]{Lemma}

\theoremstyle{definition}

\numberwithin{equation}{section}

 \title[non-homogeneous nonlocal Neumann problem]{A supercritical nonlocal Neumann problem involving non-homogeneous fractional Laplacian} 
 
\author[Remi Yvant Temgoua]{Remi Yvant Temgoua}
\address{}

\email{remi.y.temgoua@aims-senegal.org}

\date{\today\\ 	\textit{Keywords.} Fractional $g$-Laplacian, Fractional Orlicz-Sobolev spaces, Variational principle, Radial solutions, Neumann problem.\\
 \textit{~~~~~~2020 Mathematics Subject Classification:} 35A15, 35B06, 35B09, 35R11, 46E30.} 

\begin{document}

\begin{abstract}
	In this paper, we study the existence of positive non-decreasing radial solutions of a nonlocal non-standard growth problem ruled by the fractional $g$-Laplace operator with exterior Neumann condition. Our argument exploits some properties of fractional Orlicz-Sobolev spaces combined with a variational principle for nonsmooth functionals, which allows to deal with problems lacking compactness. 
\end{abstract}

\maketitle

\begin{comment}
{\footnotesize
\begin{center}

\textit{Keywords.} Fractional $g$-Laplacian, Fractional Orlicz-Sobolev spaces, Variational principle, Radial solutions, Neumann problem.\\[0.2cm]

%\textit{MSC2020:} 35R11, 45C05.

\end{center}
\textit{~~~~~~2020 Mathematics Subject Classification:} 35A15, 35B06, 35B09, 35J60. 
}
\end{comment}

\section{Introduction and main results}\label{section:introduction}

We consider the following non-homogeneous nonlocal Neumann problem
 
\begin{equation}\label{e2}
     \left\{\begin{aligned}
     (-\Delta_g)^s u+u&=|u|^{p-2}u~~~\text{in}~~B_1 \\
     \cN_gu&=0 \quad\quad\quad~~\text{in}~~\R^N\setminus\overline{B}_1 
     \end{aligned}
     \right.
	\end{equation}
where $p>2$, $s\in(0,1)$, and $B_1$ is the unit ball of $\R^N$ with $N\geq2$. Here, $(-\Delta_g)^s$ is the so-called fractional $g$-Laplacian introduced in \cite{bonder2019fractional} and defined as
\begin{equation}\label{non-homogeneous-fractional-laplacian}
(-\Delta_g)^su(x)=P.V.\int_{\R^N}g\Bigg(\frac{|u(x)-u(y)|}{|x-y|^{s}}\Bigg)\frac{u(x)-u(y)}{|u(x)-u(y)|}\frac{dy}{|x-y|^{N+s}},
\end{equation}
where ''$P.V$'' stands for the Cauchy principle value and $G$ is a Young function such that $g=G'$. Moreover, $\cN_g$ is the nonlocal Neumann condition introduced in \cite{bahrouni2021neumann} defined by 
\begin{equation}
    \cN_gu(x)=\int_{B_1}g\Bigg(\frac{|u(x)-u(y)|}{|x-y|^{s}}\Bigg)\frac{u(x)-u(y)}{|u(x)-u(y)|}\frac{dy}{|x-y|^{N+s}}\quad\quad\text{for all}~~x\in\R^N\setminus\overline{B}_1. 
\end{equation}
In recent years, problems involving non-standard growth non-local operator like \eqref{non-homogeneous-fractional-laplacian} have received a huge attention. This operator is the natural one to consider when studying nonlocal problems with a behavior more general than a power, see for instance \cite{alberico2021fractionall,alberico2021fractional,azroul2020existence,bahrouni2021neumann,bonder2023homogeneous,chaker2022regularity,salort2020eigenvalues}. Observe that in the case when $G(t)=\frac{t^p}{p}, p>1$, the operator \eqref{non-homogeneous-fractional-laplacian} reduces to the fractional $p$-Laplacian. In this case, problem \eqref{e2} was studied recently in \cite{cinti2020nonlocal,cinti2023existence} for the standand fractional Laplacian $(-\Delta)^s$. See also \cite{amundsen2023radial} where the case of mixed operator of the form $-\Delta+(-\Delta)^s$ was considered.

Regarding the existence of solutions to problem \eqref{e2}, very few references can be found in the literature. To the best of our knowledge, the only paper investigating on the existence of solution to \eqref{e2} is \cite{bahrouni2021neumann}. Notice that in this paper, only nonlinearities with subcritical growths were considered. It is therefore natural to ask if solution can be obtained when dealing with nonlinearity of supercritical type. A positive answer to this question was recently given in \cite{kouhestani2021existence} where the corresponding local problem as $s=1$ of \eqref{e2}, namely
\begin{equation}\label{local}
\left\{
\begin{aligned}
-\div\Big(g(|\nabla u|)\frac{\nabla u}{|\nabla u|}\Big)+u&=|u|^{p-2}u~~~\text{in}~B_1\\
\frac{\partial u}{\partial\nu}&=0~~~~~~~~~~~\text{on}~\partial B_1
\end{aligned}
\right.
\end{equation}
was considered. Precisely, the authors in \cite{kouhestani2021existence} showed that the non-homogeneous Neumann problem \eqref{local} possesses at least one radially increasing solution. They also addressed the nonconstancy of solution. The main tool in their approach is the use of new variational principle introduced in \cite{moameni2017variational} which is based on  the critical point theory of Szulkin for non-smooth functionals \cite{szulkin1986minimax} that allows one to deal with supercritical problems variationally by limiting the corresponding functional on a proper convex subset instead of the whole space. Notice that such a convex set (in which compactness is recovered) consists, in general, of functions enjoying certain symmetry and monotonicity properties.

However, to the best of the author's knowledge, there are no results available in the literature dealing with existence issues in the context of nonlocal non-homogeneous Neumann problem like \eqref{e2} in the supercritical regime. The present paper aims to extend the results of \cite{kouhestani2021existence} to the nonlocal framework of \eqref{e2}, giving a much more general equivalence result for nonlocal non-homogeneous Neumann problem. In this context, one has to consider some theory of fractional Orlicz-Sobolev spaces. The main result we prove is the existence of at least one radially non-decreasing solution. This is achieved under some structural assumption on $G$. 

In order to state our main results, we shall first introduce some properties of Young function. Throughout the paper, the Young function
\begin{equation}\label{young}
G(t)=\int_{0}^{t}g(\tau)\ d\tau
\end{equation}
is assumed to satisfy the growing condition
\begin{equation}\label{g}
1<q^-\leq\frac{tg(t)}{G(t)}\leq q^+<\infty~~~\forall t>0,
\end{equation}
where
\begin{equation*}
q^{-}:=\inf_{t>0}\frac{tg(t)}{G(t)} \quad\quad\text{and}\quad\quad q^+:=\sup_{t>0}\frac{tg(t)}{G(t)}.
\end{equation*}
Roughly speaking, condition \eqref{g} tells that $G$ remains between two power functions. In \cite[Theorem 4.1]{kranosel1961convex} (see also \cite[Theorem 3.4.4]{kufner1979function}) it is shown that the upper bound in \eqref{g} is equivalent to the so-called \textit{$\Delta_2$ condition (or doubling condition)}, namely 
\begin{equation*} 
\quad\quad\quad\quad\quad\qquad g(2t)\leq 2^{q^{+}-1}g(t),\quad\quad\quad G(2t)\leq 2^{q^+}G(t)~~~~~~t\geq0.\qquad\qquad\qquad\qquad\qquad~~~(\Delta_2)     
\end{equation*}
We shall also assume the following property on $G$:
\begin{equation}\label{g2}
\text{the function}~~t\mapsto G(\sqrt{t}),~~t\in[0,\infty)~\text{is convex}.
\end{equation}
With these preliminaries, our first result reads as follows.
\begin{thm}\label{first-main-result}
	Let $G$ be the Young function define in \eqref{young}. Assume that \eqref{g2} holds, $p>2$, and $2\leq q^{-}\leq q^{+}<p$. Then problem \eqref{e2} admits at least one positive non-decreasing radial solution. 
\end{thm}
Our second main result deals with the nonconstancy of solution. It reads as follows. 
\begin{thm}\label{second-main-result}
	If $2<q^{-}$, then problem \eqref{e2} admits at least one non-constant non-decreasing radial solution for every $q^{+}<p$.
	
	In the case when $q^{-}=2$, let's introduce 
	\begin{equation*}
	\Lambda:=\inf\Bigg\{2\iint_{\cQ}G\Bigg(\frac{|u(x)-u(y)|}{|x-y|^s}\Bigg)\frac{dxdy}{|x-y|^{N}}: u\in \Sigma\Bigg\}
	\end{equation*}
	where $\Sigma:=\{u\in\cX_{rad}(B_1): u~\text{is non-decreasing with respect to}~r=|x|,~\int_{B_1}u\ dx=0,~\int_{B_1}|u|^2\ dx=1\}$. If $q^{+}<p$ satisfies
	\begin{equation}
	\Big(\frac{p}{2}\Big)^{\frac{q^{+}-2}{p-2}}\Lambda<(p-2),
	\end{equation}
	then problem \eqref{e2} admits at least one non-constant non-decreasing radial solution. 
\end{thm}
We prove the existence result stated in Theorem \ref{first-main-result} by applying a vriational principle introduced in \cite{moameni2017variational} toghether with a variant of Mountain Pass Theorem due to Szulkin \cite{szulkin1986minimax}. \\

The paper is organized as follows. In Section \ref{section:prliminary}, we first introduce some notations. Next, we also introduce the notion of Young function and state some useful inequalities; the section ends with the presentation of a variational principle that allows us to obtain the existence of solution. Section \ref{section:proof-of-main-result} is devoted to the proof of Theorem \ref{first-main-result} while in Section \ref{section:non-constancy}, we prove Theorem \ref{second-main-result}.

\section{Notation and Preliminaries}\label{section:prliminary}
In this section, we gather some preliminary definitions and technical results that will be useful in the forthcoming sections. After introducing some notations, we recall the notion of Young function and present some simple technical inequalities that will be helpful. Next, we introduce the suitable fractional Orlicz-Sobolev spaces where the variational setting is stated. Finally, we also present an abstract existence result.

\subsection{Notations}
Troughout the paper, we use the notation
\begin{equation*}
u^+(x)=\max\{u(x),0\}~~~\text{and}~~~u^-(x)=\max\{-u(x),0\}.
\end{equation*}
Moreover, $\omega_N=|B_1|$ denotes the volume of the unit ball $B_1$ of $\R^N$. Finally, we denote by $\Omega^c:=\R^N\setminus\Omega$ the complementary of any subset $\Omega$ of $\R^N$.

\subsection{Young functions}
An application $G:\R_+\to\R_+$ is called Young function if it has the integral representation
\begin{equation*}
G(t)=\int_{0}^{t}g(\tau)\ d\tau
\end{equation*}
where the right-continuous function $g:\R_+\to\R_+$ satisfies the following properties
\begin{itemize}
	\item [$(i)$] $g(0)=0,~ g(t)>0$ for $t>0$
	
	\item [$(ii)$] $g$ is non-decreasing on $(0,\infty)$
	
	\item [$(iii)$] $\lim\limits_{t\to\infty}g(t)=\infty$.
\end{itemize}
The following convex property will be useful.
\begin{lemma}\label{lm4}\cite[Lemma 2.1]{lamperti1958isometries}
	Let $G$ be a Young function satisfying \eqref{g} and \eqref{g2}. Then for all $a, b\in\R$,
	\begin{equation*}
	\frac{G(|a|)+G(|b|)}{2}\geq G\Big(\Big|\frac{a+b}{2}\Big|\Big)+G\Big(\Big|\frac{a-b}{2}\Big|\Big).
	\end{equation*}
\end{lemma}
We also have the following well-known properties on Young functions, see \cite[Lemma 2.1]{bahrouni2021neumann}.
\begin{lemma}\label{lm5}
	Let $G$ be a Young function satisfying \eqref{g} and let $\xi^{-}(t)=\min\{t^{q^-},t^{q^+}\}$, $\xi^{+}(t)=\max\{t^{q^-},t^{q^+}\}$ for all $t\geq0$. Then for all $a,b\geq0$,
	\begin{itemize}
		\item [$(i)$] $\xi^{-}(a)G(b)\leq G(ab)\leq\xi^{+}(a)G(b)$
		
		\item [$(ii)$] $G(a+b)\leq 2^{q^{+}}(G(a)+G(b))$
		
		\item [$(iii)$] $G$ is Lipschitz continuous. 
	\end{itemize} 
\end{lemma}
We conclude this subsection by mentioning some examples of functions $g$ for which the corresponding Young functions satisfy the assumptions of our main results. 
\begin{itemize}
	\item [$(1)$] $g(t)=r|t|^{r-2}t$, for all $t\in\R$ with $2\leq r<p$. 
	
	\item [$(2)$] $g(t)=r|t|^{r-2}t\log(1+|t|)+\frac{|t|^{r-1}t}{1+|t|}$, for all $t\in\R$, with $2\leq r<p$.
	
	\item [$(3)$] $g(t)=r_1|t|^{r_1-2}t+r_2|t|^{r_2-2}t$, for all $t\in\R$ with $2\leq r_1<r_2<p$.
\end{itemize}

\subsection{Fractional Orlicz-Sobolev spaces}
Let $\Omega$ be an open subset of $\R^N$. Given a Young function $G$, we let  $\cX(\Omega)$ be the fractional Orlicz-Sobolev space define as
\begin{equation}\label{orlicz-space}
\cX(\Omega):=\Bigg\{u:\R^N\to\R~\text{measurable}:u|_{\Omega}\in L^2(\Omega)~\text{and}~\iint_{\cQ}G\Bigg(\frac{|u(x)-u(y)|}{|x-y|^s}\Bigg)\frac{dxdy}{|x-y|^{N}}<\infty\Bigg\},
\end{equation}
where $\cQ=\R^{2N}\setminus(\Omega^c)^2$. Notice that $\cX(\Omega)$ is a reflexive Banach space with respect to the norm
\begin{equation*}
\|u\|_{\cX(\Omega)}:=\|u\|_{L^2(\Omega)}+[u]_{s,G,*}
\end{equation*}
where
\begin{equation*}
[u]_{s,G,*}:=\inf\Bigg\{\lambda>0: \iint_{\cQ}G\Bigg(\frac{|u(x)-u(y)|}{\lambda|x-y|^s}\Bigg)\frac{dxdy}{|x-y|^{N}}\leq 1\Bigg\}
\end{equation*}
is the so-called Luxemburg seminorm.
 
The following result can be found in \cite[Lemma 3.1]{bahrouni2020basic} (see also \cite[Lemma 2.1]{fukagai2006positive}). 
\begin{lemma}\label{lm3}
	Let $G$ be a Young function satisfying \eqref{g} and let $\xi^{-}(t)=\min\{t^{q^-},t^{q^+}\}$, $\xi^{+}(t)=\max\{t^{q^-},t^{q^+}\}$ for all $t\geq0$. Then,
	\begin{equation}
	\xi^{-}([u]_{s,G,*})\leq\rho_{s,G,*}(u)\leq\xi^{+}([u]_{s,G,*})~~\text{for all}~u\in\cX(B_1), 
	\end{equation}
	where
	\begin{equation*}
	\rho_{s,G,*}(u)=\iint_{\cQ}G\Bigg(\frac{|u(x)-u(y)|}{|x-y|^s}\Bigg)\frac{dxdy}{|x-y|^{N}}. 
	\end{equation*}
\end{lemma}

\begin{remark}\label{rmk1}
	As mentioned in the introduction, condition \eqref{g} indicates that the Young function $G$ lies between two power functions. Precisely, we have 
	\begin{equation*}
	c_1t^{q^-}\leq G(t)\leq c_2t^{q^+}~~~\text{
	for all}~t\geq1,
	\end{equation*}
	where $c_1$ and $c_2$ are two positive constants. In particular, the following embeddings can be obtained
	\begin{equation*}
	\cX(\Omega)\hookrightarrow W^{s,q^{-}}_*(\Omega),~~~~~~~~~~~W^{s,q^{+}}_*(\Omega)\hookrightarrow\cX(\Omega)
	\end{equation*}
	where, for every $r>1$, $W^{s,r}_*(\Omega)$ is a fractional Sobolev space defined as 
	\begin{equation*}
	W^{s,r}_*(\Omega):=\Bigg\{u:\R^N\to\R~\text{measurable}: u|_{\Omega}\in L^2(\Omega)~\text{and}~\iint_{\cQ}\frac{|u(x)-u(y)|^r}{|x-y|^{N+rs}}\ dxdy<\infty\Bigg\}.  
	\end{equation*}
\end{remark}

\subsection{Moameni's variational approach}
We present in this subsection a more general variational principle due to Moameni \cite{moameni2017variational} for non-smooth functionals from which our existence results follow.

Let $V$ be a real Banach space and denote by $V^*$ its topological dual. Let $K$ be a non-empty convex and closed subset of $V$. Moreover, let  $\Psi: V\to (-\infty,+\infty]$ be proper, convex, and lower semi-continuous which is G$\hat{\text{a}}$teaux differentiable on $K$. Consider $\Phi\in C^1(V,\R)$ and let $I: V\to(-\infty,\infty]$ be the functional define by 
\begin{equation}\label{functional}
I=\Psi-\Phi.
\end{equation}
Let $I_K=\Psi_K-\Phi$ be the restriction of $I$ on $K$ where
\begin{equation*}
\Psi_K(u)=\left\{
\begin{aligned}
&\Psi(u)~~~\text{if}~~u\in K\\
&\infty~~~~~~~\text{otherwise}.
\end{aligned}
\right.
\end{equation*}
We now recall the following definition of a critical point for lower semi-continuous functions due to Szulkin, see \cite{szulkin1986minimax}.

\begin{defi}\label{def3}
	A point $\ov u\in V$ is a critical point of $I=\Psi-\Phi$ if $u\in Dom(\Psi)$ and it satisfies the inequality
	\begin{equation}\label{critical-point}
	\langle D\Phi(\ov u), \ov u-v\rangle+\Psi(v)-\Psi(\ov u)\geq0~~~\text{for all}~~v\in V.
	\end{equation}
\end{defi}
We also recall the following notion of pointwise invariance condition from \cite{moameni2017variational}.

\begin{defi}\label{point-wise-condition}
	The triple $(\Psi,K,\Phi)$ is said to satisfies the pointwise invariance condition at a point $\ov u\in V$ if there exists a convex  G$\hat{\text{a}}$teaux differentiable function $H: V\to\R$ at a point $\ov v\in K$ suh that
	\begin{equation*}
	D\Psi(\ov v)+DH(\ov v)=D\Phi(\ov u)+DH(\ov u).
	\end{equation*}
\end{defi}
With the above definitions at hand, we can now state the following new variational principle established in \cite{moameni2017variational}, which is the main tool in the proof of Theorem \ref{first-main-result}.
\begin{thm}\label{variational-theorem}
	Let $V$ be a reflexive Banach space and $K$ be convex and weakly closed subset of $V$. Let $\Psi: V\to (-\infty,+\infty]$ be proper, convex and lower semi-continuous which is G$\hat{\text{a}}$teaux differentiable on $K$ and let $\Phi\in C^1(V,\R)$. Assume that the following two assumptions hold.
	\begin{itemize}
		\item [$(i)$] The functional $I_K: V\to(-\infty,+\infty]$ defined by $I_K=\Psi_K-\Phi$ has a critical point $\ov u\in V$ in the sense of Definition \ref{def3}, and;
		
		\item [$(ii)$] The triple $(\Psi,K,\Phi)$ satisfies the pointwise invariance condition at the point $\ov u$.
	\end{itemize}
Then $\ov u\in K$ is a solution of the equation
\begin{equation*}
D\Psi(u)=D\Phi(u).
\end{equation*}
\end{thm}
We close this section with the following minimax principal. We start with the following definition. It is due to Szulkin \cite{szulkin1986minimax}.
\begin{defi}
	Let $c\in\R$. We say that $I=\Psi-\Phi$ satisfies the  Palais-Smale compactness condition (PS) at $c$, if every sequence $\{u_n\}$ such that $I(u_n)\to c$ and 
	\begin{equation}
	\langle D\Phi(u_n), u_n-v\rangle+\Psi(v)-\Psi(u_n)\geq-\varepsilon_n\|u_n-v\|_{V}~~~\text{for all}~~v\in V,
	\end{equation}
	where $\varepsilon_n\to0$, possesses a convergent subsequence.
\end{defi}
A variant of Mountain Pass Theorem established in \cite{szulkin1986minimax} reads as follows. Notice that it plays a key role in establishing assertion $(i)$ of Theorem \ref{variational-theorem}.

\begin{thm}(Mountain pass theorem)\label{mountain-pass-theorem}
	Suppose that $I:V\to (-\infty,+\infty]$ is as in \eqref{functional} and satisfies (PS) and suppose moreover that $I$ satisfies mountain pass geometry (MPG)
	\begin{itemize}
		\item [$(i)$] $I(0)=0$
		\item[$(ii)$] there exists $e\in V$ such that $I(e)\leq 0$
		\item[$(iii)$] there exists some $\eta$ such that $0<\eta<\|e\|$ and for every $u\in V$ with $\|u\|=\eta$ one has $I(u)>0$.
	\end{itemize}
	Then $I$ has a critical value $c$ which is defined by
	\begin{equation}\label{critical-value}
	c=\inf_{\gamma\in\Gamma}\sup_{t\in [0,1]}I(\gamma(t)),
	\end{equation}
	where $\Gamma=\{\gamma\in C([0,1],V): \gamma(0)=0,~ \gamma(1)=e\}$.
\end{thm}

\section{Proof of Theorem \ref{first-main-result}}\label{section:proof-of-main-result}
In this section, we aim to establish our first main result. To this end, we first prove a simplified version of Theorem \ref{variational-theorem} applicable to problem \eqref{e2}. We then proceed with the proof of our main result by showing that all the assumptions of Theorem \ref{mountain-pass-theorem} are fulfilled.

Before doing that, we first present the variational setting in our context. Let $\cX(B_1)$ be the fractional Orlicz-Sobolev space define in \eqref{orlicz-space} and denote by $\cX_{rad}(B_1)$ the space of radial functions in $\cX(B_1)$. Define
\begin{equation*}
V=\cX_{rad}(B_1)\cap L^p(B_1).
\end{equation*}
Notice that $V$ is a Banach space equipped with the norm
\begin{equation*}
\|\cdot\|_V=\|\cdot\|_{\cX(B_1)}+\|\cdot\|_{L^p(B_1)}.
\end{equation*}
The formal  Euler-Lagrange functional  $I:V\to(-\infty,+\infty]$ of \eqref{e2} is defined by
\begin{equation*}
I(u)=\iint_{\cQ}G\Bigg(\frac{|u(x)-u(y)|}{|x-y|^s}\Bigg)\frac{dxdy}{|x-y|^{N}}+\frac{1}{2}\int_{B_1}|u(x)|^2\ dx-\frac{1}{p}\int_{B_1}|u(x)|^p\ dx. 
\end{equation*}
Next, define $\Psi:V\to\R$ and $\Phi:V\to\R$ by 
\begin{equation*}
\Psi(u)=\iint_{\cQ}G\Bigg(\frac{|u(x)-u(y)|}{|x-y|^s}\Bigg)\frac{dxdy}{|x-y|^{N}}+\frac{1}{2}\int_{B_1}|u(x)|^2\ dx~~\text{and}~~\Phi(u)=\frac{1}{p}\int_{B_1}|u(x)|^p\ dx
\end{equation*}
so that $$I=\Psi-\Phi.$$
Notice that $\Phi\in C^1(V,\R)$ and $\Psi$ is a proper, convex, and lower semi-continuous function which is G$\hat{\text{a}}$teaux differentiable, that is, $\langle D\Psi(u), \phi\rangle$ exists for all $u,\phi\in V$ with
\begin{equation*}
\langle D\Psi(u),\phi\rangle=\iint_{\cQ}g\Bigg(\frac{|u(x)-u(y)|}{|x-y|^s}\Bigg)\frac{u(x)-u(y)}{|u(x)-u(y)|}\frac{\phi(x)-\phi(y)}{|x-y|^s}\frac{dxdy}{|x-y|^{N}}+\int_{B_1}u\phi\ dx. 
\end{equation*}
We now introduce the convex weakly closed subset $K$ of $V$ as
\begin{equation}\label{convex-set}
K:=\{u\in V: u\geq0, u~\text{is radially non-decreasing with respect to}~r=|x|\}.
\end{equation}
Let $I_K:=\Psi_K-\Phi$ be the restriction of $I$ to $K$ where
\begin{equation*}
\Psi_K(u):=\left\{\begin{aligned}
&\Psi(u)~~~\text{if}~~u\in K,\\
&+\infty~~~\text{otherwise}.
\end{aligned}
\right.
\end{equation*}
Here and throughout the rest of the paper, we assume that the Young function $G(t)=\int_{0}^{t}g(\tau)$ satisfies 

\begin{equation}\label{g'}
2\leq q^-=\inf_{t>0}\frac{tg(t)}{G(t)} \leq\frac{tg(t)}{G(t)}\leq q^+=\sup_{t>0}\frac{tg(t)}{G(t)}<p~~~\forall t>0,
\end{equation}
and that 
\begin{equation}\label{g2'}
\text{the function}~t\mapsto G(\sqrt{t})~\text{is convex on}~[0,\infty).
\end{equation}
Inspired by a variational principle in \cite{moameni2018variational}, we prove the following simplified version of Theorem \ref{variational-theorem} applicable to our problem.
\begin{thm}\label{t1}
	Let $V=\cX_{rad}(B_1)\cap L^p(B_1)$, and let $K$ be the convex and weakly closed subset of $V$ defined in \eqref{convex-set}. Suppose the following two assertions hold:
	\begin{itemize}
		\item [$(i)$] The functional $I_K$ has a critical point $\overline{u}\in V$ in the sense of Definition \ref{def3}, and;
		\item [$(ii)$] there exist $\overline{v}\in K$ with $\cN_g\overline{v}=0$ in $\R^N\setminus\overline{B}_1$ such that
		\begin{equation}\label{v-tilde-equation}
		(-\Delta_g)^s\overline{v}+\overline{v}=D\Phi(\overline{u})=|\overline{u}|^{p-2}\overline{u},
		\end{equation}
		in the weak sense namely,
		\begin{align}\label{v-tilde-weak-formulation}
		\iint_{\cQ}g\Bigg(\frac{|\overline{v}(x)-\overline{v}(y)|}{|x-y|^s}\Bigg)\frac{\overline{v}(x)-\overline{v}(y)}{|\overline{v}(x)-\overline{v}(y)|}\frac{\phi(x)-\phi(y)}{|x-y|^{s}}\frac{dxdy}{|x-y|^N}+\int_{B_1}\overline{v}\phi\ dx=\int_{B_1}D\Phi(\overline{u})\phi\ dx,~\forall\phi\in V.
		\end{align}
	\end{itemize}
	Then $\overline{u}\in K$ is a weak solution of the equation
	\begin{equation}\label{u1}
	(-\Delta_g)^su+u=|u|^{p-2}u 
	\end{equation}
	with Neumann condition. 
\end{thm}
\begin{proof}
	Since by assumption $(i)$ $\overline{u}$ is a critical point of $I_K$, then by Definition \ref{def3}, we have
	\begin{align*}
	&\iint_{\cQ}G\Bigg(\frac{|\phi(x)-\phi(y)|}{|x-y|^s}\Bigg)\frac{dxdy}{|x-y|^{N}}-\iint_{\cQ}G\Bigg(\frac{|\overline{u}(x)-\overline{u}(y)|}{|x-y|^s}\Bigg)\frac{dxdy}{|x-y|^{N}}\\
	&~~~+\frac{1}{2}\int_{B_1}|\phi(x)|^2\ dx-\frac{1}{2}\int_{B_1}|\overline{u}(x)|^2\ dx\geq\int_{B_1}|\overline{u}|^{p-2}\overline{u}(\phi-\overline{u})\ dx, \quad\quad\forall\phi\in K.
	\end{align*}
	By taking in particular $\phi=\overline{v}$, the above inequality becomes
	\begin{align}\label{l1}
	\nonumber&\iint_{\cQ}G\Bigg(\frac{|\overline{v}(x)-\overline{v}(y)|}{|x-y|^s}\Bigg)\frac{dxdy}{|x-y|^{N}}-\iint_{\cQ}G\Bigg(\frac{|\overline{u}(x)-\overline{u}(y)|}{|x-y|^s}\Bigg)\frac{dxdy}{|x-y|^{N}}\\
	&~~~~~~~+\frac{1}{2}\int_{B_1}|\overline{v}(x)|^2\ dx-\frac{1}{2}\int_{B_1}|\overline{u}(x)|^2\ dx\geq\int_{B_1}|\overline{u}|^{p-2}\overline{u}(\overline{v}-\overline{u})\ dx.
	\end{align}
	We now use $\phi=\overline{v}-\overline{u}$ as a test function in \eqref{v-tilde-weak-formulation} to get
	\begin{align}\label{l2}
	\nonumber\iint_{\cQ}g\Bigg(\frac{|\overline{v}(x)-\overline{v}(y)|}{|x-y|^s}\Bigg)\frac{\overline{v}(x)-\overline{v}(y)}{|\overline{v}(x)-\overline{v}(y)|}\frac{(\overline{v}-\overline{u})(x)-(\overline{v}-\overline{u})(y)}{|x-y|^{s}}&\frac{dxdy}{|x-y|^N}+\int_{B_1}\overline{v}(\overline{v}-\overline{u})\ dx\\
	&=\int_{B_1}|\overline{u}|^{p-2}\overline{u}(\overline{v}-\overline{u})\ dx. 
	\end{align}
	Plugging \eqref{l2} into \eqref{l1}, we get
	\begin{align}\label{l3}
	\nonumber\frac{1}{2}\int_{B_1}|\overline{u}-\overline{v}|^2\ dx&\leq\iint_{\cQ}G\Bigg(\frac{|\overline{v}(x)-\overline{v}(y)|}{|x-y|^s}\Bigg)\frac{dxdy}{|x-y|^{N}}-\iint_{\cQ}G\Bigg(\frac{|\overline{u}(x)-\overline{u}(y)|}{|x-y|^s}\Bigg)\frac{dxdy}{|x-y|^{N}}\\
	&-\iint_{\cQ}g\Bigg(\frac{|\overline{v}(x)-\overline{v}(y)|}{|x-y|^s}\Bigg)\frac{\overline{v}(x)-\overline{v}(y)}{|\overline{v}(x)-\overline{v}(y)|}\frac{(\overline{v}-\overline{u})(x)-(\overline{v}-\overline{u})(y)}{|x-y|^{s}}\frac{dxdy}{|x-y|^N}.
	\end{align}
	Since $t\mapsto G(\sqrt{t}) $ is convex, then 
	\begin{equation}\label{property-of-covex-functions}
	G(\sqrt{t_2})-G(\sqrt{t_1})\geq \frac{1}{2\sqrt{t_1}}g(\sqrt{t_1})(t_2-t_1)~~~~\text{for all}~~t_1, t_2\geq0.
	\end{equation}
	By substituting $t_1=\Big(\frac{\overline{v}(x)-\overline{v}(y)}{|x-y|^s}\Big)^2$ and $t_2=\Big(\frac{\overline{u}(x)-\overline{u}(y)}{|x-y|^s}\Big)^2$ in \eqref{property-of-covex-functions}, we have 
	\begin{align}\label{z0}
\nonumber	G\Bigg(\frac{|\overline{u}(x)-\overline{u}(y)|}{|x-y|^s}\Bigg)&-G\Bigg(\frac{|\overline{v}(x)-\overline{v}(y)|}{|x-y|^s}\Bigg)\\
\nonumber	&\geq\frac{1}{2} g\Bigg(\frac{|\overline{v}(x)-\overline{v}(y)|}{|x-y|^s}\Bigg)\frac{|x-y|^s}{|\ov v(x)-\ov v(y)|}\Bigg[\Big(\frac{\overline{u}(x)-\overline{u}(y)}{|x-y|^s}\Big)^2-\Big(\frac{\overline{v}(x)-\overline{v}(y)}{|x-y|^s}\Big)^2\Bigg]\\
	&= \frac{1}{2} g\Bigg(\frac{|\overline{v}(x)-\overline{v}(y)|}{|x-y|^s}\Bigg)\frac{|x-y|^s}{|\ov v(x)-\ov v(y)|} A_{\ov u, \ov v}(x,y),
	\end{align}
	where 
	\begin{align}\label{z1}
\nonumber	A_{\ov u, \ov v}(x,y)&=\Big(\frac{\overline{u}(x)-\overline{u}(y)}{|x-y|^s}\Big)^2-\Big(\frac{\overline{v}(x)-\overline{v}(y)}{|x-y|^s}\Big)^2\\
\nonumber	&=\frac{(\ov v(x)-\ov v(x))(\ov u(x)-\ov u(y)-(\ov v(x)-\ov v(y)))}{|x-y|^{2s}}\\
	&~~~~~~~~~~+\frac{(\ov u(x)-\ov u(x))(\ov u(x)-\ov u(y)-(\ov v(x)-\ov v(y)))}{|x-y|^{2s}}.
	\end{align}
Utilizing the elementary Young inequality $ab\leq\frac{1}{2}a^2+\frac{1}{2}b^2$, the second term on the right-hand side of the above equation can be bound from below as
\begin{align*}
\frac{(\ov u(x)-\ov u(x))(\ov u(x)-\ov u(y)-(\ov v(x)-\ov v(y)))}{|x-y|^{2s}}&=\frac{(\ov u(x)-\ov u(x))^2-(\ov u(x)-\ov u(y))(\ov v(x)-\ov v(y))}{|x-y|^{2s}}\\
&\geq\frac{1}{2}\frac{(\ov u(x)-\ov u(y))^2}{|x-y|^{2s}}-\frac{1}{2}\frac{(\ov v(x)-\ov v(y))^2}{|x-y|^{2s}}\\
&=\frac{1}{2}A_{\ov u, \ov v}(x,y).
\end{align*}
From this, it follows from \eqref{z1} that
	\begin{align*}
	A_{\ov u, \ov v}(x,y)\geq2\frac{(\ov v(x)-\ov v(x))(\ov u(x)-\ov u(y)-(\ov v(x)-\ov v(y)))}{|x-y|^{2s}}.
\end{align*}
Plugging this into \eqref{z0}, there holds that
\begin{align}\label{z3}
\nonumber G\Bigg(\frac{|\overline{u}(x)-\overline{u}(y)|}{|x-y|^s}\Bigg)&-G\Bigg(\frac{|\overline{v}(x)-\overline{v}(y)|}{|x-y|^s}\Bigg)\\
&\geq g\Bigg(\frac{|\overline{v}(x)-\overline{v}(y)|}{|x-y|^s}\Bigg)\frac{\ov v(x)-\ov v(y)}{|\ov v(x)-\ov v(y)|}\frac{\ov u(x)-\ov u(y)-(\ov v(x)-\ov v(y))}{|x-y|^{s}}.
\end{align}
Thus, multiplying \eqref{z3} by $\frac{1}{|x-y|^{N}}$ and integrating over $\cQ$, one gets  
	\begin{align*}
	&\iint_{\cQ}G\Bigg(\frac{|\overline{u}(x)-\overline{u}(y)|}{|x-y|^s}\Bigg)\frac{dxdy}{|x-y|^{N}}-\iint_{\cQ}G\Bigg(\frac{|\overline{v}(x)-\overline{v}(y)|}{|x-y|^s}\Bigg)\frac{dxdy}{|x-y|^{N}}\\
	&~~~~~~\geq\iint_{\cQ} g\Bigg(\frac{|\overline{v}(x)-\overline{v}(y)|}{|x-y|^s}\Bigg)\frac{\ov v(x)-\ov v(y)}{|\ov v(x)-\ov v(y)|}\frac{\ov u(x)-\ov u(y)-(\ov v(x)-\ov v(y))}{|x-y|^{s}}\frac{dxdy}{|x-y|^{N}}, 
	\end{align*}
	 from which we deduce that the right-hand side of \eqref{l3} is nonpositive and thus 
	\begin{align*}
	\int_{B_1}|\overline{u}-\overline{v}|^2\ dx=0.
	\end{align*}
	In particular, $\overline{v}-\overline{u}=0$ i.e., $\overline{v}=\overline{u}$ a.e., in $B_1$. We then deduce from \eqref{v-tilde-equation} that $\overline{u}$ satisfies in the weak sense equation \eqref{u1}. Moreover, since $\overline{v}$ satisfies Neumann conditions, so does $\overline{u}$. The proof is therefore finished. 
\end{proof}
Remark that assertion $(ii)$ in Theorem \ref{t1} tells us that the triple $(\Psi,K,\Phi)$ satisfies the pointwise invariance condition of Theorem \ref{variational-theorem} at $\ov u$ given in Definition \ref{point-wise-condition}. Here, $H=0$ is the corresponding convex G$\hat{\text{a}}$teaux diffentiable function. 

We now wish to verify that all the assumptions of the Mountain Pass Theorem (see Theorem \ref{mountain-pass-theorem}) are satisfy. Before, let us first establish the following lemma which is of key importance.
\begin{lemma}\label{equivalence-of-norms}
	There exists a positive constant $C>0$ such that
	\begin{equation*}\label{lm2}
	\|u\|_{\cX(B_1)}\leq\|u\|_V\leq C\|u\|_{\cX(B_1)},~~~\forall u\in K.
	\end{equation*}
\end{lemma} 
The principal ingredient in the proof of the above lemma is the following radial lemma. 
\begin{lemma}\label{radial-lemma}
	Let $u\in L^2(B_1)$ be a nonnegative radial non-decreasing function. Then there exists a positive constant $C>0$ such that
	\begin{equation*}
	\|u\|_{L^{\infty}(B_1)}\leq C\|u\|_{L^2(B_1)}.
	\end{equation*}
\end{lemma}

\begin{proof}
	Let $0<r<1$ and denote $B_r$ the ball centered at the origin with radius $r$. We now identify $u$ with its trivial extension. Then,
	\begin{align*}
	\|u\|^2_{L^2(B_1)}=\|u\|^2_{L^2(\R^N)}&=\int_{\R^N}|u(x)|^2\ dx\geq\int_{B_{r+\frac{1}{2}}\setminus B_r}|u(x)|^2\ dx\\
	&=\gamma_{N-1}\int_{r}^{r+\frac{1}{2}}|u(t)|^2t^{N-1}\ dt\\
	&\geq\gamma_{N-1}|u(r)|^2\int_{r}^{r+\frac{1}{2}}t^{N-1}\ dt\\
	&=\frac{\gamma_{N-1}}{N}\Big((r+\frac{1}{2})^N-r^N\Big)|u(r)|^2,
	\end{align*}
	where in the inequality above, we have used the monotonicity of $u$. Here, $\gamma_{N-1}$ denotes the surface measure of the sphere $\mathbb{S}^{N-1}$. Since $(r+\frac{1}{2})^N-r^N\geq2^{-N}$, the proof is completed. 
\end{proof}
We now give the proof of Lemma \ref{equivalence-of-norms}.

\begin{proof}[Proof of Lemma \ref{equivalence-of-norms}]
	By definition of $\|\cdot\|_V$, the first inequality follows. Now, for all $u\in K$,
	\begin{align*}
	\|u\|_V&=\|u\|_{\cX(B_1)}+\|u\|_{L^p(B_1)}\\
	&\leq\|u\|_{\cX(B_1)}+\omega_N^{\frac{1}{p}}\|u\|_{L^{\infty}(B_1)}\\
	&\leq\|u\|_{\cX(B_1)}+c\|u\|_{L^2(B_1)}\\
	&\leq C\|u\|_{\cX(B_1)},
	\end{align*} 
	where in the second inequality, we have used Lemma \ref{radial-lemma}. 
\end{proof}
In the next two lemmas, we prove that the functional $I_K$ satisfies the Palais-Small condition and has the geometric features needed to apply the Mountain Pass theorem. 
\begin{lemma}\label{mpg}
	If $2\leq q^-\leq q^+<p$, then the functional $I_K$ satisfies the Mountain Pass Geometry (MPG).
\end{lemma}
\begin{proof}
	It suffices to show that $I_K$ fulfills conditions $(i), (ii)$, and $(iii)$ of Theorem \ref{mountain-pass-theorem}.  Clearly, $I_K$ satisfies condition $(i)$ since $I_K(0)=0$. Now, for $c\in\R_+$,
	\begin{align*}
	I_K(c)&=\iint_{\cQ}G\Bigg(\frac{|c-c)|}{|x-y|^s}\Bigg)\frac{dxdy}{|x-y|^{N}}+\frac{1}{2}\int_{B_1}c^2\ dx-\frac{1}{p}\int_{B_1}c^p\ dx\\
	&=\omega_N\Big(\frac{c^2}{2}-\frac{c^p}{p}\Big),
	\end{align*}
	where $\omega_N=|B_1|$. Since $2<p$, then $I_K(c)\leq0$ for $c$ sufficiently large. So, condition $(ii)$ follows with $e=c$ for some $0<c\in\R$ sufficiently large. For condition $(iii)$ we first notice that for $u\in V\setminus K$, $I_K(u)>0$ by definition. Finally, let $u\in K$ with $\|u\|_V=\eta$ for some $0<\eta<1$ sufficiently small. Then, by Lemma \ref{lm3}, 
	\begin{align*}
	I_K(u)&=\rho_{s,G,*}(u)+\frac{1}{2}\|u\|^2_{L^2(B_1)}-\frac{1}{p}\|u\|^p_{L^p(B_1)}\\
	&\geq \xi^{-}([u]_{s,G,*})+\frac{1}{2}\|u\|^2_{L^2(B_1)}-\frac{1}{p}\|u\|^p_{L^p(B_1)}\\
	&=[u]_{s,G,*}^{q^+}+\frac{1}{2}\|u\|^2_{L^2(B_1)}-\frac{1}{p}\|u\|^p_{L^p(B_1)}.
	\end{align*}
	Since $2\leq q^+$,
	\begin{align*}
	I_K(u)&\geq\frac{1}{2}\Big([u]_{s,G,*}^{q^+}+\|u\|^{q^+}_{L^2(B_1)}\Big)-\frac{1}{p}\|u\|^p_{L^p(B_1)}\\
	&\geq\frac{1}{2^{q^+}}\|u\|^{q^+}_{\cX(B_1)}-\frac{1}{p}\|u\|^p_{L^p(B_1)}\\
	&\geq\frac{C}{2^{q^+}}\|u\|^{q^+}_{V}-\frac{1}{p}\|u\|^p_{V},
	\end{align*}
	where in the latter, we have used Lemma \ref{lm2}. Thus,
	\begin{equation}\label{l}
	I_K(u)\geq C_1\eta^{q^+}-C_2\eta^p,
	\end{equation}
	where $C_1$ and $C_2$ are two positive constants. Since $q^+<p$, then for $0<\eta<1$ sufficiently small, the right-hand side of \eqref{l} is strictly positive and thus $I_K(u)>0$. 
\end{proof}

\begin{lemma}\label{ps}
	If $2\leq q^-\leq q^+<p$, the the functional $I_K$ satisfies the Palais-Small compactness condition (PS).
\end{lemma}
\begin{proof}
	Let $\{u_n\}\subset K$ be a sequence such that $I_K(u_n)\to c\in \R$ and
	\begin{equation}\label{a1}
	\langle D\Phi(u_n), u_n-v\rangle+\Psi(v)-\Psi(u_n)\geq-\varepsilon_n\|u_n-v\|_{V},~~~\text{for all}~~v\in V.
	\end{equation}
	The aim is to show that $\{u_n\}$ possesses a convergent subsequence. For that, since $I_K(u_n)\to c$, then for $n$ sufficiently large,
	\begin{equation}\label{a}
	\iint_{\cQ}G\Bigg(\frac{|u_n(x)-u_n(y)|}{|x-y|^s}\Bigg)\frac{dxdy}{|x-y|^{N}}+\frac{1}{2}\int_{B_1}|u_n(x)|^2\ dx-\frac{1}{p}\int_{B_1}|u_n(x)|^p\ dx\leq c+1.
	\end{equation}
	Before going further, we introduce the function $f(t)=t^{q^+}-p(t-1)-1$ for $t\in(1,\infty)$. Then there exists $t_0:=\Big(\frac{p}{q^+}\Big)^{\frac{1}{q^{+}-1}}$ such that for all $t\in(1,t_0)$, $f(t)<0$  i.e., $t^{q^+}-1<p(t-1)$. In the rest of the proof we take $t\in(1,t_0)$. Notice that for such $t$, there also holds that $t^2-1\leq t^{q^{+}}-1$ since $2\leq q^+$.
	
	Now, substituting $v=tu_n$ in \eqref{a1} and recalling that $\langle D\Phi(u_n),u_n\rangle=\int_{B_1}|u_n(x)|^p\ dx$, then
	\begin{align}\label{a2}
	\nonumber&\iint_{\cQ}G\Bigg(\frac{|u_n(x)-u_n(y)|}{|x-y|^s}\Bigg)\frac{dxdy}{|x-y|^{N}}-\iint_{\cQ}G\Bigg(t\frac{|u_n(x)-u_n(y)|}{|x-y|^s}\Bigg)\frac{dxdy}{|x-y|^{N}}\\
	&+(t-1)\int_{B_1}|u_n(x)|^p\ dx+\frac{1-t^2}{2}\int_{B_1}|u_n(x)|^2\ dx\leq\varepsilon_n(t-1)\|u_n\|_V.
	\end{align}
	Since $t>1$ then from Lemma \ref{lm5}-(i), we have
	\begin{equation}\label{a3}
	\iint_{\cQ}G\Bigg(t\frac{|u_n(x)-u_n(y)|}{|x-y|^s}\Bigg)\frac{dxdy}{|x-y|^{N}}\leq t^{q^+}\iint_{\cQ}G\Bigg(\frac{|u_n(x)-u_n(y)|}{|x-y|^s}\Bigg)\frac{dxdy}{|x-y|^{N}}.
	\end{equation}
	So, \eqref{a2} and \eqref{a3} yield
	\begin{align}\label{a4}
	\nonumber(1-t^{q^+})\iint_{\cQ}G\Bigg(\frac{|u_n(x)-u_n(y)|}{|x-y|^s}\Bigg)\frac{dxdy}{|x-y|^{N}}&+(t-1)\int_{B_1}|u_n(x)|^p\ dx+\frac{1-t^2}{2}\int_{B_1}|u_n(x)|^2\ dx\\
	&~~~~~~~~~~~~~~~~~\leq\varepsilon_n(t-1)\|u_n\|_V\leq C\|u_n\|_V.
	\end{align}
	Next, since $t^{q^+}-1<p(t-1)$, then take $\zeta>0$ such that $t^{q^+}-1<\zeta<p(t-1)$. Multiplying \eqref{a} by $\zeta$ and summing it up with \eqref{a4}, we get
	\begin{align*}
	&(\zeta+1-t^{q^+})\iint_{\cQ}G\Bigg(\frac{|u_n(x)-u_n(y)|}{|x-y|^s}\Bigg)\frac{dxdy}{|x-y|^{N}}+\Big(t-1-\frac{\zeta}{p}\Big)\int_{B_1}|u_n(x)|^p\ dx\\
	&~~~~~~~~~~+\frac{\zeta+1-t^2}{2}\int_{B_1}|u_n(x)|^2\ dx\leq C_0+C\|u_n\|_V,
	\end{align*}
	that is
	\begin{align}\label{a5}
	 \nonumber C_1\iint_{\cQ}G\Bigg(\frac{|u_n(x)-u_n(y)|}{|x-y|^s}\Bigg)\frac{dxdy}{|x-y|^{N}}+C_2\int_{B_1}|u_n(x)|^p\ dx&+C_3\int_{B_1}|u_n(x)|^2\ dx\\
	 &\leq C_0+C\|u_n\|_V,
	\end{align}
	where $C_1=(\zeta+1-t^{q^+}), C_2=t-1-\frac{\zeta}{p}$, and $C_3=\frac{\zeta+1-t^2}{2}$ are positive constants, thanks to the choice of $\zeta$. 
	
	Our next aim is to show that $\{u_n\}$ is bounded in $V$. We argue by contradiction. Assume that $\|u_n\|_V\to+\infty$. Then, by Lemma \ref{lm2}, $\|u_n\|_{\cX(B_1)}\to+\infty$. In particular, either $[u_n]_{s,G,*}\to+\infty$ or $\|u_n\|_{L^2(B_1)}\to+\infty$. Let us focus in the former case i.e., $[u_n]_{s,G,*}\to+\infty$. By Lemma \ref{lm3}, 
	\begin{equation}\label{a6}
	\|u_n\|^2_{L^2(B_1)}+[u_n]^{q^-}_{s,G,*}\leq\int_{B_1}|u_n(x)|^2\ dx+\iint_{\cQ}G\Bigg(\frac{|u_n(x)-u_n(y)|}{|x-y|^s}\Bigg)\frac{dxdy}{|x-y|^{N}}.
	\end{equation}
	Recalling that $2\leq q^-$, then \eqref{a5} and \eqref{a6} yield
	\begin{align}\label{a7}
	\nonumber\|u_n\|^2_{\cX(B_1)}&=(\|u_n\|_{L^2(B_1)}+[u_n]_{s,G,*})^2\leq 2(\|u_n\|^2_{L^2(B_1)}+[u_n]^{2}_{s,G,*})\leq2(\|u_n\|^2_{L^2(B_1)}+[u_n]^{q^-}_{s,G,*})\\
	\nonumber&\leq 2\Bigg(\int_{B_1}|u_n(x)|^2\ dx+\iint_{\cQ}G\Bigg(\frac{|u_n(x)-u_n(y)|}{|x-y|^s}\Bigg)\frac{dxdy}{|x-y|^{N}}\Bigg)\\
	\nonumber&\leq 2\max\Bigg\{\frac{1}{C_1},\frac{1}{C_3}\Bigg\}\Bigg(C_3\int_{B_1}|u_n(x)|^2\ dx+C_1\iint_{\cQ}G\Bigg(\frac{|u_n(x)-u_n(y)|}{|x-y|^s}\Bigg)\frac{dxdy}{|x-y|^{N}}\Bigg)\\
	&\leq C_0+C\|u_n\|_V\leq C_0+C'\|u_n\|_{\cX(B_1)}, 
	\end{align}
	where in the latter, we have used Lemma \ref{lm2}. We then deduce from \eqref{a7} that $\{u_n\}$ is bounded in $\cX(B_1)$ which is a contradiction. Thus, $\{[u_n]_{s,G,*}\}$ is bounded i.e., $[u_n]_{s,G,*}\leq M$ for some $M>0$. So,
	\begin{align*}
	C_3\|u_n\|^2_{L^2(B_1)}&\leq C_1\iint_{\cQ}G\Bigg(\frac{|u_n(x)-u_n(y)|}{|x-y|^s}\Bigg)\frac{dxdy}{|x-y|^{N}}+C_2\int_{B_1}|u_n(x)|^p\ dx+C_3\int_{B_1}|u_n(x)|^2\ dx\\
	&\leq C_0+C\|u_n\|_V\leq C_0+C_4\|u_n\|_{\cX(B_1)}\leq C_0+C_4M+C_4\|u_n\|_{L^2(B_1)}.
	\end{align*}
	In particular, $\{u_n\}$ is also bounded in $L^2(B_1)$ and thus we cannot have $\|u_n\|_{L^2(B_1)}\to+\infty$. In conclusion, $\{u_n\}$ is bounded in $V$. Hence, after passing to a subsequence, there exists $\overline{u}\in V$ such that 
	\begin{equation}\label{a8}
	\begin{aligned}
	&u_n\rightharpoonup\overline{u}~~\text{weakly in}~V\\
	&u_n\to\overline{u}~~\text{strongly in}~L^2(B_1)\\
	&u_n\to\overline{u}~~\text{a.e. in}~B_1.
	\end{aligned}
	\end{equation}
	Notice that the second limit in \eqref{a8} is a direct consequence of the fact that $u_n\rightharpoonup\overline{u}$ weakly in $\cX(B_1)$ and that the embedding $\cX(B_1)\hookrightarrow L^2(B_1)$ is compact, thanks to Remark \ref{rmk1}. On the other hand, by Lemma \ref{radial-lemma} and from the the boundedness of $\{u_n\}$ in $\cX(B_1)$, we deduce that $\{u_n\}$ is also bounded in $L^{\infty}(B_1)$. In particular,
	\begin{equation}\label{a9}
	u_n\to\overline{u}~~\text{strongly in}~L^{r}(B_1) ~\text{for all}~r\geq1.
	\end{equation}
	We now wish to show that
	\begin{equation}\label{a10}
	u_n\to\overline{u}~~\text{strongly in}~V.
	\end{equation}
   By Fatou's Lemma, there holds
   \begin{equation}\label{a11}
   \iint_{\cQ}G\Bigg(\frac{|\overline{u}(x)-\overline{u}(y)|}{|x-y|^s}\Bigg)\frac{dxdy}{|x-y|^{N}}\leq\liminf_{n\to+\infty}\iint_{\cQ}G\Bigg(\frac{|u_n(x)-u_n(y)|}{|x-y|^s}\Bigg)\frac{dxdy}{|x-y|^{N}}.
   \end{equation}
	Substituting $v=\overline{u}$ in \eqref{a1}, we have
	\begin{align}\label{a12}
\nonumber	\int_{B_1}u_n^{p-1}(\overline{u}-u_n)\ dx&+\iint_{\cQ}G\Bigg(\frac{|u_n(x)-u_n(y)|}{|x-y|^s}\Bigg)\frac{dxdy}{|x-y|^{N}}-\iint_{\cQ}G\Bigg(\frac{|\overline{u}(x)-\overline{u}(y)|}{|x-y|^s}\Bigg)\frac{dxdy}{|x-y|^{N}}\\
	&+\frac{1}{2}\int_{B_1}|u_n|^2\ dx-\frac{1}{2}\int_{B_1}|\overline{u}|^2\ dx\leq\varepsilon_n\|u_n-\overline{u}\|_{V}.
	\end{align}
	From \eqref{a9} and the boundedness of $\{u_n\}$ in $L^{\infty}(B_1)$, we have
	\begin{equation*}
	\int_{B_1}u_n^{p-1}(\overline{u}-u_n)\ dx\to0~~~\text{and}~~~\frac{1}{2}\int_{B_1}|u_n|^2\ dx-\frac{1}{2}\int_{B_1}|\overline{u}|^2\ dx\to0~~\text{as}~~n\to\infty.
	\end{equation*}
	Moreover, since $\|u_n-\overline{u}\|_V$ is bounded for $n$ sufficiently large, then by passing to the limit as $n\to\infty$ in \eqref{a12}, there holds that
	\begin{equation}\label{a13}
	\liminf_{n\to\infty}\iint_{\cQ}G\Bigg(\frac{|u_n(x)-u_n(y)|}{|x-y|^s}\Bigg)\frac{dxdy}{|x-y|^{N}}\leq\iint_{\cQ}G\Bigg(\frac{|\overline{u}(x)-\overline{u}(y)|}{|x-y|^s}\Bigg)\frac{dxdy}{|x-y|^{N}}.
	\end{equation}
	Thus, from \eqref{a11} and \eqref{a13} we deduce that
		\begin{equation}\label{a14}
	\liminf_{n\to\infty}\iint_{\cQ}G\Bigg(\frac{|u_n(x)-u_n(y)|}{|x-y|^s}\Bigg)\frac{dxdy}{|x-y|^{N}}=\iint_{\cQ}G\Bigg(\frac{|\overline{u}(x)-\overline{u}(y)|}{|x-y|^s}\Bigg)\frac{dxdy}{|x-y|^{N}}.
	\end{equation}
	To complete the proof of \eqref{a10}, it remains to show that $[u_n-\overline{u}]_{s,G,*}\to0$ as $n\to\infty$. We argue by contadiction. Assume that $[u_n-\overline{u}]_{s,G,*}$ does not converges to $0$ as $n\to\infty$, that is, there exists $\varepsilon>0$ and a subsequence $\{u_{n_k}\}\subset K$ of $\{u_n\}$ such that
	\begin{equation}\label{a15}
	\Big[\frac{n_{n_k}-\overline{u}}{2}\Big]_{s,G,*}>\varepsilon.
	\end{equation}
	Then, by Lemma \ref{lm3},
	\begin{equation}\label{a16}
	\rho_{s,G,*}\Big(\frac{u_{n_k}-\overline{u}}{2}\Big)>\xi^{-}(\varepsilon),
	\end{equation}
	where $\rho_{s,G,*}(u)=\iint_{\cQ}G\Big(\frac{|u(x)-u(y)|}{|x-y|^s}\Big)\frac{dxdy}{|x-y|^{N}}$.
	
	On the other hand, using Lemma \ref{lm4} with $a=\frac{u_{n_k}(x)-u_{n_k}(y)}{|x-y|^s}$ and $b=\frac{\overline{u}(x)-\overline{u}(y)}{|x-y|^s}$, there hold that
	\begin{align*}
	\frac{1}{2}\Bigg\{G\Big(\frac{|u_{n_k}(x)-u_{n_k}(y)|}{|x-y|^s}\Big)&+G\Big(\frac{|\overline{u}(x)-\overline{u}(y)|}{|x-y|^s}\Big)\Bigg\}-G\Big(\frac{|(u_{n_k}+\overline{u})(x)-(u_{n_k}+\overline{u})(y)|}{2|x-y|^s}\Big)\\
	&\geq G\Big(\frac{|(u_{n_k}-\overline{u})(x)-(u_{n_k}-\overline{u})(y)|}{2|x-y|^s}\Big), 
	\end{align*}
	which together with \eqref{a16} yield 
	\begin{equation}\label{a17}
	\frac{1}{2}\Big(\rho_{s,G,*}(u_{n_k})+\rho_{s,G,*}(\overline{u})\Big)-\rho_{s,G,*}\Big(\frac{u_{n_k}+\overline{u}}{2}\Big)\geq \rho_{s,G,*}\Big(\frac{u_{n_{k}}-\overline{u}}{2}\Big)>\xi^{-}(\varepsilon). 
	\end{equation}
	Now, taking the $\liminf$ in \eqref{a17} and recalling \eqref{a14}, we get
	\begin{equation}\label{a18}
	\rho_{s,G,*}(\overline{u})-\xi^{-}(\varepsilon)\geq\liminf_{k\to\infty}\rho_{s,G,*}\Big(\frac{u_{n_k}+\overline{u}}{2}\Big).
	\end{equation}
	Since $\frac{u_{n_k}+\overline{u}}{2}$ converges weakly to $\overline{u}$ in $\cX_{rad}(B_1)$, and modulars are lower semicontinuous with respect to the weak convergence, we get 
	\begin{equation}\label{a19}
	\rho_{s,G,*}(\overline{u})\leq\liminf_{k\to\infty}\rho_{s,G,*}\Big(\frac{u_{n_k}+\overline{u}}{2}\Big).
	\end{equation}
	From \eqref{a18} and \eqref{a19} we get a contradiction. Thus $[u_n-\overline{u}]_{s,G,*}\to0$ as $n\to\infty$. This implies that $u_{n}\to\overline{u}$ in $\cX_{rad}(B_1)$ and thus \eqref{a10} follows. The proof is finished. 
\end{proof}
The following lemma tells us that the pointwise invariance condition $(ii)$ of Theorem \ref{variational-theorem} is satisfied.

\begin{lemma}\label{p-w-c}
	Suppose that $\overline{u}\in K$. Then there exists $v\in K$ solving 
	\begin{equation}\label{a20}
	\left\{
	\begin{aligned}
	(-\Delta_g)^sv+v&=|\overline{u}|^{p-2}\overline{u}~~\text{in}~B_1,\\
	\cN_gv&=0~~~~~~~~~\text{in}~\R^N\setminus\overline{B}_1,
	\end{aligned}
	\right.
	\end{equation}
	in the weak sense, namely,
	\begin{align}\label{v-tilde-weak-formulation-1}
	\iint_{\cQ}g\Bigg(\frac{|v(x)-v(y)|}{|x-y|^s}\Bigg)\frac{v(x)-v(y)}{|v(x)-v(y)|}\frac{\phi(x)-\phi(y)}{|x-y|^{s}}\frac{dxdy}{|x-y|^N}+\int_{B_1}v\phi\ dx=\int_{B_1}|\overline{u}|^{p-2}\overline{u}\phi\ dx, 
	\end{align}
	for all $\phi\in V$. 
\end{lemma}

\begin{proof}
	Introduce the functional $J:\cX_{rad}(B_1)\to\R$ as
	\begin{equation*}
	J(v)=\iint_{\cQ}G\Bigg(\frac{|v(x)-v(y)|}{|x-y|^s}\Bigg)\frac{dxdy}{|x-y|^{N}}+\frac{1}{2}\int_{B_1}|v(x)|^2\ dx-\int_{B_1}|\overline{u}(x)|^{p-2}\overline{u}v\ dx. 
	\end{equation*}
	Then $J$ is convex, lower-semicontinuous on $\cX_{rad}(B_1)$. Moreover, $J$ is coercive on $\cX_{rad}(B_1)$. In fact, in view of Lemma \ref{lm3}, we have
	\begin{equation*}
	J(v)\geq\xi^{-}([v]_{s,G,*})+\|v\|^2_{L^2(B_1)}-c\|v\|_{L^2(B_1)}
	\end{equation*}
	where $c>0$ is a positive constant depending on $p$. So $\lim\limits_{\|v\|_{\cX(B_1)}\to\infty}J(v)=+\infty$. In particular, $J$ achieves its minimum at some point $v\in\cX_{rad}(B_1)$. It is not difficult to see that $v$ satisfies 
	\begin{align}\label{v-weak-formulation}
	\iint_{\cQ}g\Bigg(\frac{|v(x)-v(y)|}{|x-y|^s}\Bigg)\frac{v(x)-v(y)}{|v(x)-v(y)|}\frac{\phi(x)-\phi(y)}{|x-y|^{s}}\frac{dxdy}{|x-y|^N}+\int_{B_1}v\phi\ dx=\int_{B_1}|\overline{u}|^{p-2}\overline{u}\phi\ dx, 
	\end{align}
	for all $\phi\in \cX(B_1)$. Taking $\phi\equiv0$ in $\overline{B}_1$ in \eqref{v-weak-formulation} and from the integration by parts formula (see \cite[Proposition 2.6]{bahrouni2021neumann}) there holds
	\begin{align*}
	0=	\iint_{\cQ}g\Bigg(\frac{|v(x)-v(y)|}{|x-y|^s}\Bigg)\frac{v(x)-v(y)}{|v(x)-v(y)|}\frac{\phi(x)-\phi(y)}{|x-y|^{s}}\frac{dxdy}{|x-y|^N}=\int_{\R^N\setminus\overline{B}_1}\phi(x)\cN_gv(x)\ dx 
	\end{align*}
	from which we deduce that $\cN_gv=0$ a.e. in $\R^N\setminus\overline{B}_1$. In other words, $v$ satifies the Neumann condition. To complete the proof, it remains to show that $v\in K$. 
	
	We claim that $v\geq0$. Indeed, to see this, we use $\phi=v^{-}$ as a test function in \eqref{v-weak-formulation} to get
\begin{align}\label{b}
\nonumber\iint_{\cQ}g\Bigg(\frac{|v(x)-v(y)|}{|x-y|^s}\Bigg)\frac{v(x)-v(y)}{|v(x)-v(y)|}\frac{v^{-}(x)-v^{-}(y)}{|x-y|^{s}}&\frac{dxdy}{|x-y|^N}+\int_{B_1}vv^{-}\ dx\\
&=\int_{B_1}|\overline{u}|^{p-2}\overline{u}v^{-}\ dx. 
\end{align}
Now, since $(v(x)-v(y))(v^{-}(x)-v^{-}(y))\leq-(v^{-}(x)-v^{-}(y))^2$ and $v(x)v^{-}(x)=-|v^{-}(x)|^2$, then \eqref{b} becomes
\begin{align}\label{b'}
-\iint_{\cQ}g\Bigg(\frac{|v(x)-v(y)|}{|x-y|^s}\Bigg)\frac{(v^{-}(x)-v^{-}(y))^2}{|v(x)-v(y)|}\frac{dxdy}{|x-y|^{N+s}}-\int_{B_1}|v^{-}(x)|^2\ dx\geq\int_{B_1}|\overline{u}|^{p-2}\overline{u}v^{-}\ dx. 
\end{align}
Since the right-hand side of \eqref{b'} is non-negative, we then deduce that
\begin{equation*}
\int_{B_1}|v^{-}(x)|^2\ dx=0,
\end{equation*}
which implies in particular that $v^{-}(x)=0$ for a.e. $x\in B_1$. So, $v\geq0$ in $B_1$. This completes the proof of the claim.

We now show that $v$ is non-decreasing with respect to the radial variable. We borrow some ideas in \cite{amundsen2023radial} (see also \cite{Colasuonno2017supercritical}). Fix $r_0\in(0,1)$. It suffices to show that for every $r\in(r_0,1)$, one of the following situation occurs
\begin{itemize}
	\item [$(i)$] $v(t)\leq v(r)$  for all $t\in(r_0,r)$
	
	\item [$(ii)$] $v(t)\geq v(r)$  for all $t\in(r,1)$.
\end{itemize}
If $|\overline{u}(r)|^{p-2}\overline{u}(r)\leq v(r)$, we use
\begin{equation*}
\phi(x)=\left\{
\begin{aligned}
&(v(|x|)-v(r))^{+}~~\text{if}~~r_0<|x|\leq r\\
&0~~~~~~~~~~~~~~~~~~~~~\text{otherwise}
\end{aligned}
\right.
\end{equation*}
as a test function in \eqref{v-weak-formulation} to see that
\begin{align*}
&\iint_{\R^{2N}\setminus((B_r\setminus B_{r_0})^c)^2}g\Bigg(\frac{|v(x)-v(y)|}{|x-y|^s}\Bigg)\frac{v(x)-v(y)}{|v(x)-v(y)|}\frac{(v(|x|)-v(r))^{+}-(v(|y|)-v(r))^{+}}{|x-y|^{s}}\frac{dxdy}{|x-y|^N}\\
&~~~~~~~~~~~~~~+\int_{B_r\setminus B_{r_0}}v(x)(v(|x|)-v(r))^{+}\ dx=\int_{B_r\setminus B_{r_0}}|\overline{u}(x)|^{p-2}\overline{u}(x)(v(|x|)-v(r))^{+}\ dx.
\end{align*}
Since $v(x)(v(|x|)-v(r))^{+}=|(v(|x|)-v(r))^{+}|^2+v(r)(v(|x|)-v(r))^{+}$, the above equation becomes 
\begin{align}\label{b5}
\nonumber&\iint_{\R^{2N}\setminus((B_r\setminus B_{r_0})^c)^2}g\Bigg(\frac{|v(x)-v(y)|}{|x-y|^s}\Bigg)\frac{v(x)-v(y)}{|v(x)-v(y)|}\frac{(v(|x|)-v(r))^{+}-(v(|y|)-v(r))^{+}}{|x-y|^{s}}\frac{dxdy}{|x-y|^N}\\
\nonumber&~~~~~+\int_{B_r\setminus B_{r_0}}|(v(|x|)-v(r))^{+}|^2\ dx=\int_{B_r\setminus B_{r_0}}(|\overline{u}(x)|^{p-2}\overline{u}(x)-v(r))(v(|x|)-v(r))^{+}\ dx\\
&~~~~~~~~~~~~~~~~~~~~~~~~~~~~~~~~~~~~~~~~~~~~~~\leq\int_{B_r\setminus B_{r_0}}(|\overline{u}(r)|^{p-2}\overline{u}(r)-v(r))(v(|x|)-v(r))^{+}\ dx,
\end{align}
where in the latter, we have used the monotonicity of $\overline{u}$. Since, $(v(x)-v(y))((v(|x|)-v(r))^{+}-(v(|y|)-v(r))^+)\geq ((v(|x|)-v(r))^{+}-(v(|y|)-v(r))^+)^2$ then the first term on the left-hand side of \eqref{b5} is non-negative. Since by assumption the right-hand side of \eqref{b5} is non-positive, we deduce that
$$\int_{B_r\setminus B_{r_0}}|(v(|x|)-v(r))^{+}|^2\ dx=0$$
from which we get in particular that $(v(|x|)-v(r))^{+}=0$ for all $r_0<|x|\leq r$. Thus, $(i)$ holds.

Similarly, when $|\overline{u}(r)|^{p-2}\overline{u}(r)> v(r)$, we use
\begin{equation*}
\phi(x)=\left\{
\begin{aligned}
&0~~~~~~~~~~~~~~~~~~~~~~~~\text{if}~~r_0<|x|\leq r\\
&(v(|x|)-v(r))^{-}~~~~~\text{otherwise} 
\end{aligned}
\right.
\end{equation*}
as a test function in \eqref{v-weak-formulation} to see that
\begin{align*}
&\iint_{\R^{2N}\setminus(B_r^c)^2}g\Bigg(\frac{|v(x)-v(y)|}{|x-y|^s}\Bigg)\frac{(v(|x|)-v(r))^{-}-(v(|y|)-v(r))^{-}}{|x-y|^{s}}\frac{dxdy}{|x-y|^N}\\
&~~~~~~~~~~~~~~+\int_{B_r^c}v(x)(v(|x|)-v(r))^{-}\ dx=\int_{B_r^c}|\overline{u}(x)|^{p-2}\overline{u}(x)(v(|x|)-v(r))^{-}\ dx. 
\end{align*} 
Since $v(x)(v(|x|)-v(r))^{-}=-|(v(|x|)-v(r))^{-}|^2+v(r)(v(|x|)-v(r))^{-}$, then it follows from the above equation that
\begin{align}\label{b6}
\nonumber&\iint_{\R^{2N}\setminus(B_r^c)^2}g\Bigg(\frac{|v(x)-v(y)|}{|x-y|^s}\Bigg)\frac{v(x)-v(y)}{|v(x)-v(y)|}\frac{(v(|x|)-v(r))^{-}-(v(|y|)-v(r))^{-}}{|x-y|^{s}}\frac{dxdy}{|x-y|^N}\\
\nonumber&~~~~~-\int_{B_r^c}|(v(|x|)-v(r))^{-}|^2\ dx=\int_{B_r^c}(|\overline{u}(x)|^{p-2}\overline{u}(x)-v(r))(v(|x|)-v(r))^{-}\ dx\\
&~~~~~~~~~~~~~~~~~~~~~~~~~~~~~~~~~~~~~~~~~~\geq\int_{B_r^c}(|\overline{u}(r)|^{p-2}\overline{u}(r)-v(r))(v(|x|)-v(r))^{-}\ dx. 
\end{align}
Note that in the latter, we have used again the monotonicity of $\overline{u}$. Now, since $(v(x)-v(y))((v(|x|)-v(r))^{-}-(v(|y|)-v(r))^{-})\leq -((v(|x|)-v(r))^{-}-(v(|y|)-v(r))^{-})^2$ then the first term on the left-hand side of \eqref{b6} is non-positive. On the other hand, since by assumption the right-hand side of \eqref{b6} is non-negative, we then deduce that $$\int_{B_r^c}|(v(|x|)-v(r))^{-}|^2\ dx=0,$$ and thus $(v(|x|)-v(r))^{-}=0$ for $r<|x|<1$. So, $(ii)$ follows.  
\end{proof}
We are now ready to prove our first main result.
\begin{proof}[Proof of Theorem \ref{first-main-result} (completed)]
	In view of Lemmas \ref{mpg} and \ref{ps}, the functional $I_K$ satisfies the Palais-Small condition and possesses the Mountain Pass Geometry. Thus, from the Mountain Pass Theorem (see Theorem \ref{mountain-pass-theorem}), $I_K$ admits a critical point $\ov u\in V$ in the sense of Definition \ref{def3}. On the other hand, the triple $(\Psi,K,\Phi)$ satisfies the pointwise invariance condition at $\ov u$, thanks to Lemma \ref{p-w-c}. Thus, by Theorem \ref{variational-theorem}, we deduce that $\ov u$ is a weak solution of \eqref{e2}.   
\end{proof}

\section{Proof of Theorem \ref{second-main-result}}\label{section:non-constancy}
The goal of this section is to prove Theorem \ref{second-main-result}.

\begin{proof}[Proof of Theorem \ref{second-main-result}]
	The existence of a solution $\ov u\in K$ of \eqref{e2} has been proved in Theorem \ref{first-main-result}. We now show that $\ov u$ is not constant. Notice first that problem \eqref{e2} has only two constant solutions, $0$ and $1$, and, from its energy level, we know that $\ov u\not\equiv 0$. Let us show now that $\overline{u}\not\equiv 1$. To this end, we first recall that $I(\overline{u})=I_K(\overline{u})=c$ where $c$ is defined (see \eqref{critical-value}) by
	\begin{equation}
	c=\inf_{\gamma\in\Gamma}\sup_{t\in [0,1]}I(\gamma(t)),
	\end{equation}
	where $\Gamma=\{\gamma\in C([0,1],V): \gamma(0)=0\neq\gamma(1), I(\gamma(1))\leq0\}$. So, to get $\overline{u}\not\equiv 1$, it suffices to show that $c<I_K(1)$.

	Let $v\in\cX_{rad}(B_1)$ be a non-constant, non-decreasing function. Then by Lemm \ref{radial-lemma}, $v\in L^{\infty}(B_1)$. Take $\tau>0$ such that $\tau<\frac{1}{\|v\|_{L^{\infty}(B_1)}}$. Notice that for such $\tau$, there holds that $r(1+\tau v)\in K$ for all $r>0$. Now, for $0<\tau<1$ and $r>1$ we define $\gamma_{\tau,r}:[0,1]\to V$ by $\gamma_{\tau,r}(t)=r(1+\tau v)t$.
	
	By Lemma \ref{lm5}-(i),
\begin{align}\label{d}
\nonumber	I_K(\gamma_{\tau,r}(t))&=\iint_{\cQ}G\Bigg(\frac{|(r(1+\tau v)t)(x)-(r(1+\tau v)t)(y)|}{|x-y|^s}\Bigg)\frac{dxdy}{|x-y|^{N}}\\
\nonumber	&+\frac{1}{2}\int_{B_1}|(r(1+\tau v)t)(x)|^2\ dx-\frac{1}{p}\int_{B_1}|(r(1+\tau v)t)(x)|^p\ dx\\
\nonumber	&=\iint_{\cQ}G\Bigg(r\tau t\frac{| v(x)- v(y)|}{|x-y|^s}\Bigg)\frac{dxdy}{|x-y|^{N}}+\frac{r^2t^2}{2}\int_{B_1}|((1+\tau v))(x)|^2\ dx\\
\nonumber	&~~~~~~~~~~~~~~~~~~~~~~~-\frac{r^pt^p}{p}\int_{B_1}|((1+\tau v))(x)|^p\ dx\\
	&\leq r^{q^{+}}(\tau t)^{q^{-}}\iint_{\cQ}G\Bigg(\frac{| v(x)- v(y)|}{|x-y|^s}\Bigg)\frac{dxdy}{|x-y|^{N}}+\frac{r^2t^2}{2}\int_{B_1}|((1+\tau v))(x)|^2\ dx\\
\nonumber	&~~~~~~~~~~~~~~~~~~~~~~~~-\frac{r^pt^p}{p}\int_{B_1}|((1+\tau v))(x)|^p\ dx\\
\nonumber	&\leq\frac{t^2}{2}\Bigg(2r^{q^{+}}\tau^{q^{-}}\iint_{\cQ}G\Bigg(\frac{| v(x)- v(y)|}{|x-y|^s}\Bigg)\frac{dxdy}{|x-y|^{N}}+r^2\int_{B_1}|((1+\tau v))(x)|^2\ dx\Bigg)\\
\nonumber	&~~~~~~~~~~~~~~~~~~~~~~-\frac{t^p}{p}\Bigg(r^p\int_{B_1}|((1+\tau v))(x)|^p\ dx\Bigg).
	\end{align}
Thus,
\begin{equation*}
\max_{t\in[0,1]}I_K(\gamma_{\tau,r}(t))\leq\max_{t\in[0,1]}\phi(t)
\end{equation*}
where
\begin{align*}
\phi(t)&=\frac{t^2}{2}\Bigg(2r^{q^{+}}\tau^{q^{-}}\iint_{\cQ}G\Bigg(\frac{| v(x)- v(y)|}{|x-y|^s}\Bigg)\frac{dxdy}{|x-y|^{N}}+r^2\int_{B_1}|((1+\tau v))(x)|^2\ dx\Bigg)\\
&~~~~~~~~~~-\frac{t^p}{p}\Bigg(r^p\int_{B_1}|((1+\tau v))(x)|^p\ dx\Bigg).
\end{align*}
Notice that $\phi$ attains its maximum at 
\begin{equation*}
t_0=\Bigg(\frac{2r^{q^{+}}\tau^{q^{-}}\iint_{\cQ}G\Big(\frac{| v(x)- v(y)|}{|x-y|^s}\Big)\frac{dxdy}{|x-y|^{N}}+r^2\int_{B_1}|((1+\tau v))(x)|^2\ dx}{r^p\int_{B_1}|((1+\tau v))(x)|^p\ dx}\Bigg)^{\frac{1}{p-2}}
\end{equation*}
with
\begin{equation*}
\phi(t_0)=\Bigg(\frac{1}{2}-\frac{1}{p}\Bigg)\frac{\Big(2r^{q^{+}}\tau^{q^{-}}\iint_{\cQ}G\Big(\frac{| v(x)- v(y)|}{|x-y|^s}\Big)\frac{dxdy}{|x-y|^{N}}+r^2\int_{B_1}|((1+\tau v))(x)|^2\ dx\Big)^{\frac{p}{p-2}}}{\Big(r^p\int_{B_1}|((1+\tau v))(x)|^p\ dx\Big)^{\frac{2}{p-2}}}.
\end{equation*}
So,
\begin{align*}
\max_{t\in[0,1]}I_K(\gamma_{\tau,r}(t))&\leq\max_{t\in[0,1]}\phi(t)=\phi(t_0)\\
&=\Big(\frac{1}{2}-\frac{1}{p}\Big)\frac{\Big(2r^{q^{+}}\tau^{q^{-}}\iint_{\cQ}G\Big(\frac{| v(x)- v(y)|}{|x-y|^s}\Big)\frac{dxdy}{|x-y|^{N}}+r^2\int_{B_1}|((1+\tau v))(x)|^2\ dx\Big)^{\frac{p}{p-2}}}{\Big(r^p\int_{B_1}|((1+\tau v))(x)|^p\ dx\Big)^{\frac{2}{p-2}}}. 
\end{align*}
We claim that there exist $0<\tau<1$ and $r>1$ such that
\begin{equation}\label{claim}
\max_{t\in[0,1]}I_K(\gamma_{\tau,r}(t))<I_K(1)=\Big(\frac{1}{2}-\frac{1}{p}\Big)\omega_N
\end{equation}
where $\omega_N=|B_1|$. To see this, it is enough to show that
\begin{equation*}
\Big(\frac{1}{2}-\frac{1}{p}\Big)\frac{\Big(2r^{q^{+}}\tau^{q^{-}}\iint_{\cQ}G\Big(\frac{| v(x)- v(y)|}{|x-y|^s}\Big)\frac{dxdy}{|x-y|^{N}}+r^2\int_{B_1}|((1+\tau v))(x)|^2\ dx\Big)^{\frac{p}{p-2}}}{\Big(r^p\int_{B_1}|((1+\tau v))(x)|^p\ dx\Big)^{\frac{2}{p-2}}}<I_K(1)=\Big(\frac{1}{2}-\frac{1}{p}\Big)\omega_N,
\end{equation*}
for some $0<\tau<1$ and $r>1$, that is,
\begin{align*}
&\Bigg(2r^{q^{+}}\tau^{q^{-}}\iint_{\cQ}G\Big(\frac{| v(x)- v(y)|}{|x-y|^s}\Big)\frac{dxdy}{|x-y|^{N}}+r^2\int_{B_1}|((1+\tau v))(x)|^2\ dx\Bigg)^{p}\\
&~~~~~~~~~~~~~~~~~~~<\omega_N^{p-2}\Bigg(r^p\int_{B_1}|((1+\tau v))(x)|^p\ dx\Bigg)^{2}.
\end{align*}
Let $h:\R\to\R$ be the function given by 
\begin{align*}
h(\tau)&=\Bigg(2r^{q^{+}}\tau^{q^{-}}\iint_{\cQ}G\Big(\frac{| v(x)- v(y)|}{|x-y|^s}\Big)\frac{dxdy}{|x-y|^{N}}+r^2\int_{B_1}|((1+\tau v))(x)|^2\ dx\Bigg)^{p}\\
&~~~~~~~~~~~~~~~~~~~-\omega_N^{p-2}\Bigg(r^p\int_{B_1}|((1+\tau v))(x)|^p\ dx\Bigg)^{2}.
\end{align*}
Note that $h$ is at least twice differentiable, $h(0)=0$ and $h'(0)=0$. Moreover,
\begin{align}\label{d1}
\nonumber h''(\tau)=2pr^{2p}\Bigg[&(p-1)\Bigg(2r^{q^{+}-2}\tau^{q^{-}}\iint_{\cQ}G\Big(\frac{| v(x)- v(y)|}{|x-y|^s}\Big)\frac{dxdy}{|x-y|^{N}}+\int_{B_1}|1+\tau v|^2 \Bigg)^{p-2}\\
\nonumber&\times\Bigg(r^{q^{+}-2}q^{-}\tau^{q^{-}-1}\iint_{\cQ}G\Big(\frac{| v(x)- v(y)|}{|x-y|^s}\Big)\frac{dxdy}{|x-y|^{N}}+\int_{B_1}(1+\tau v)v\ dx \Bigg)^2\\
\nonumber&+\Bigg(2r^{q^{+}-2}\tau^{q^{-}}\iint_{\cQ}G\Big(\frac{| v(x)- v(y)|}{|x-y|^s}\Big)\frac{dxdy}{|x-y|^{N}}+\int_{B_1}|1+\tau v|^2\ dx\Bigg)^{p-1}\\
\nonumber&\times\Bigg(r^{q^{+}-2}q^{-}(q^{-}-1)\tau^{q^{-}-2}\iint_{\cQ}G\Big(\frac{| v(x)- v(y)|}{|x-y|^s}\Big)\frac{dxdy}{|x-y|^{N}}+\int_{B_1}v^2\ dx\Bigg)\\
&-\omega_N^{p-2}\Bigg( \Big(\int_{B_1}|1+\tau v|^{p-1}v\Big)^2+\Big(\int_{B_1}|1+\tau v|^p\Big)\Big((p-1)\int_{B_1}|1+\tau v|^{p-2}v^2\Big) \Bigg)
\Bigg].
\end{align}
If $q^{-}>2$, then
\begin{equation*}
h''(0)=2p(p-2)r^{2p}\omega_N^{p-2}\Bigg(\Big(\int_{B_1}v\ dx\Big)^2-\omega_N\int_{B_1}v^2\ dx\Bigg).
\end{equation*}
H\"{o}lder inequality yields $h''(0)<0$. In particular, $h$ admits a local maximum at $\tau=0$. Consequently, $h(\tau)<0$ for $\tau>0$ small enough. Furthermore, from \eqref{d},
\begin{align*}
I_K(\gamma_{\tau,r}(1))&\leq r^{q^{+}}\tau^{q^{-}}\iint_{\cQ}G\Bigg(\frac{| v(x)- v(y)|}{|x-y|^s}\Bigg)\frac{dxdy}{|x-y|^{N}}+\frac{r^2}{2}\int_{B_1}|((1+\tau v))(x)|^2\ dx\\
&~~~~~~~~~~~~~~~~~~~~~~~~-\frac{r^p}{p}\int_{B_1}|((1+\tau v))(x)|^p\ dx.
\end{align*}
Since $2<q^{-}<p$, then there exists $r>1$ such that $I_K(\gamma_{\tau,r}(1))\leq0$ and thus $\gamma_{\tau,r}(1)\in\Gamma$. This completes the proof of claim \eqref{claim}. So,
\begin{equation*}
c\leq\max_{t\in[0,1]}I_K(\gamma_{\tau,r}(t))<I_K(1),
\end{equation*}
as desired.

Now, if $q^{-}=2$, we first notice that by definition of $\Lambda$, for $\Lambda_0>\Lambda$ satisfying
\begin{equation}\label{d2}
\Big(\frac{p}{2}\Big)^{\frac{q^{+}-2}{p-2}}\Lambda_0<(p-2),
\end{equation}
there exists a non-decreasing function $v\in\cX_{rad}(B_1)$ such that $\int_{B_1}v\ dx=0$, $\int_{B_1}|v|^2\ dx=1$ and $\iint_{\cQ}G\Big(\frac{| v(x)- v(y)|}{|x-y|^s}\Big)\frac{dxdy}{|x-y|^{N}}<\frac{\Lambda_0}{2}$. Now, from \eqref{d1}, there holds that
\begin{align*}
h''(0)&=2pr^{2p}\omega_N^{p-1}\Bigg(2r^{q^{+}-2}\iint_{\cQ}G\Big(\frac{| v(x)- v(y)|}{|x-y|^s}\Big)\frac{dxdy}{|x-y|^{N}}-(p-2)\Bigg)\\
&\leq2pr^{2p}\omega_N^{p-1}(r^{q^{+}-2}\Lambda_0-(p-2)).
\end{align*}
We now wish to show that $h''(0)<0$. To this end, it is enough to show that there exists $r>1$ such that
\begin{equation}\label{d3}
r^{q^{+}-2}\Lambda_0<(p-2)
\end{equation}
with $I_K(\gamma_{\tau,r}(1))\leq0$ so that $\gamma_{\tau,r}\in\Gamma$.

From \eqref{d2}, one can find some $r$ such that $\Big(\frac{p}{2}\Big)^{\frac{1}{p-2}}<r<\Big(\frac{p-2}{\Lambda_0}\Big)^{\frac{1}{q^{+}-2}}$ and thus \eqref{d3} follows. This implies that for such $r$, $h''(0)<0$. On the other hand, from \eqref{d}
\begin{align}\label{d4}
\nonumber I_K(\gamma_{\tau,r}(1))&\leq r^{q^{+}}\tau^{2}\iint_{\cQ}G\Bigg(\frac{| v(x)- v(y)|}{|x-y|^s}\Bigg)\frac{dxdy}{|x-y|^{N}}+\frac{r^2}{2}\int_{B_1}|((1+\tau v))(x)|^2\ dx\\
\nonumber&~~~~~~~~~~~~~~~~~~~~~~~~-\frac{r^p}{p}\int_{B_1}|((1+\tau v))(x)|^p\ dx\\
&\leq r^2\Big(\frac{\tau^2}{2}(r^{q^{+}-2}\Lambda_0+1)+\frac{\omega_N}{2}\Big)-\frac{r^p}{p}\int_{B_1}|((1+\tau v))(x)|^p\ dx.
\end{align}
Now, H\"{o}lder inequality with exponents $p$ and $p'=\frac{p}{p-1}$ yields
\begin{equation*}
\omega_N=\int_{B_1}(1+\tau v)\ dx\leq\omega_N^{\frac{1}{p'}}\Big(\int_{B_1}|1+\tau v|^p\ dx\Big)^{\frac{1}{p}}
\end{equation*}
that is
\begin{equation}\label{d5}
\omega_N\leq\int_{B_1}|1+\tau v|^p\ dx.
\end{equation}
From \eqref{d5} and \eqref{d4}, it follows that
\begin{align}\label{d6}
\nonumber I_K(\gamma_{\tau,r}(1))&\leq r^2\Big(\frac{\tau^2}{2}(r^{q^{+}-2}\Lambda_0+1)+\frac{\omega_N}{2}\Big)-\frac{r^p}{p}\omega_N\\
&\leq r^2\Big(\frac{\tau^2}{2}(p-1)-\Big(\frac{r^{p-2}}{p}-\frac{1}{2}\Big)\omega_N\Big)
\end{align}
where in the latter, we have used \eqref{d3}. Notice that $\frac{r^{p-2}}{p}-\frac{1}{2}>0$. Now, choosing $\tau$ sufficiently small so that
\begin{equation*}
\tau\leq\sqrt{\frac{2\Big(\frac{r^{p-2}}{p}-\frac{1}{2}\Big)}{p-1}\omega_N},
\end{equation*}
we obtain that the right-hand side of \eqref{d6} is non-positive, that is, $I_K(\gamma_{\tau,r}(1))\leq0$. In particular, $\gamma_{\tau,r}\in\Gamma$.  The proof is therefore finished. 
\end{proof}

\begin{remark}\label{rmk2}
	We would like to mention that all the results stated in this paper hold true for a more general operators of the form 
	\begin{equation*}
	\cL_{g,s}u(x):=P.V.\int_{\R^N}g\Bigg(\frac{|u(x)-u(y)|}{k(x,y)}\Bigg)\frac{u(x)-u(y)}{|u(x)-u(y)|}\frac{dy}{k(x,y)|x-y|^{N}},
	\end{equation*}
	where $k:\R^N\times\R^N\to[0,\infty]$ is a measurable kernel that satisfies
	\begin{itemize}
		\item [$(1)$] Symmetry: $k(x,y)=k(y,x)$~~~for all~$x,y\in\R^N$.
		
		\item [$(2)$] Translation invariance: $k(x+z,y+z)=k(x,y)$~~~for all~$x,y,z\in\R^N$.
		
		\item [$(3)$] Growth condition: $\Lambda^{-1} |x-y|^s\leq k(x,y)\leq\Lambda |x-y|^s$~~for all~$x,y\in\R^N$ and for some $\Lambda\geq1$. 
	\end{itemize}
Note that when $k(x,y)=|x-y|^s$, we recover the fractional $g$-Laplacian $(-\Delta_g)^s$ defined in \eqref{non-homogeneous-fractional-laplacian}. 
\end{remark}

\section*{Data availability statement}
Data sharing not applicable to this article as no datasets were generated or analyzed during the current study.

\section*{Declaration of competing interest}

The author has no competing interests to declare that are relevant to the content of this article.\\
~

\textbf{Acknowledgements:}  The author was supported by Fields Institute. Part of this reseacrh was done when the author visited the Research Group in Partial Differential Equations and its Applications of the University of Douala-Cameroon, led by Dr. Ignace Aristide Minlend. I thank the University for their hospitality.

%\textbf{Acknowledgements:} D.A. and A.M. are  pleased to acknowledge the support of the Natural Sciences and Engineering Research Council of Canada. R.Y.T. is supported by Fields Institute.

\bibliographystyle{ieeetr}

\end{document}